\documentclass{amsart}

\usepackage{color}
\usepackage{amsmath,amssymb}
\usepackage{amscd}
\usepackage{mathrsfs}
\usepackage{amsthm}
\theoremstyle{plain}

\newtheorem{theorem}{\bf Theorem}[section]
\newtheorem{lemma}[theorem]{\bf Lemma}
\newtheorem{proposition}[theorem]{\bf Proposition}
\newtheorem{corollary}[theorem]{\bf Corollary}
\newtheorem{conjecture}[theorem]{\bf Conjecture}

\theoremstyle{definition}
\newtheorem{definition}[theorem]{\bf Definition}

\newtheorem{nonexample}[theorem]{\bf Non-example}
\newtheorem{remark}[theorem]{\bf Remark}

\newcommand{\R}{\mathbb{R}}

\newcommand{\N}{\mathbb{N}}

\newcommand{\disp}{\displaystyle}

\newcommand{\nai}[2]{\langle #1,#2\rangle}

\newcommand{\eqa}[1]{
\begin{align*}
#1
\end{align*}}

\usepackage[all]{xy}

\title{Notes on a conjecture by Paszkiewicz on an ordered product of positive contractions}
\date{\today}

\author[H.~Ando]{Hiroshi Ando}
\address{Hiroshi~Ando, Department of Mathematics and Informatics, Chiba University, 1-33 Yayoi-cho, Inage, Chiba, 263- 8522,
Japan}
\email{hiroando@math.s.chiba-u.ac.jp}

\let\origmaketitle\maketitle
\def\maketitle{
  \begingroup
  \def\uppercasenonmath##1{} 
  \let\MakeUppercase\relax 
  \origmaketitle
  \endgroup
}
\begin{document}
\maketitle
\begin{abstract}
    Paszkiewicz's conjecture asserts that given a decreasing sequence $T_1\ge T_2\ge \dots$ of positive contractions on a separable infinite-dimensional Hilbert space $H$, the product $S_n=T_nT_{n-1}\cdots T_1$ converges in the strong operator topology. 
    In these notes, we give an equivalent, more precise formulation of his conjecture. Moreover, we show that the conjecture is true for the following two cases: (1) $1$ is not in the essential spectrum of $T_n$ for some $n\in \N$. (2) The von Neumann algebra generated by $\{T_n\mid n\in \N\}$ admits a faithful normal tracial state. We also remark that the analogous conjecture for the weak convergence is true. 
\end{abstract}

\noindent

\medskip

\noindent

\medskip
\section{Introduction and Statement of the result}
In the problem session of 2018 workshop ``Noncommutative Harmonic Analysis'', at B\c{e}dlewo, Paszkiewicz announced the following conjecture about a product of positive contractions. 

\begin{conjecture}[Adam Paszkiewicz, 2018]\label{conj Paszkiewicz} 
Let $H$ be a separable infinite-dimensional Hilbert space, $T_1\ge T_2\ge \dots $ be a sequence of positive linear contractions on $H$. 
Then the sequence $S_n:=T_nT_{n-1}\cdots T_1$ converges strongly. 
\end{conjecture}

In fact, it is possible to make a guess about what the limit of $S_n$ should be, if it exists. 
Note that since $T_1\ge T_2\ge $ is a non-increasing sequence of positive operators, the limit $T:=\lim_{n\to \infty}T_n$ (SOT) exists (SOT stands for the strong operator topology). 
We will use the notation that for a Borel subset $A$ of $\R$, $1_A(T)$ denotes the spectral projection of $T$ corresponding to $A$. Let $P:=1_{\{1\}}(T)$. Now consider the following

\begin{conjecture}\label{conj strong Paszkiewicz} Let $T_1\ge T_2\ge \dots $ be as in Conjecture \ref{conj Paszkiewicz}. Then
$\disp \lim_{n\to \infty}S_n=P$ ($*$-strongly). 
\end{conjecture}

The purpose of these notes is to show that the above two conjectures are actually equivalent. Moreover, we verify the conjecture for some class of operators:
\begin{theorem}\label{thm main}
    The following statements hold. 
    \begin{list}{}{}
        \item[{\rm (1)}]     Conjecture \ref{conj Paszkiewicz} and Conjecture \ref{conj strong Paszkiewicz} are equivalent.
        \item[{\rm (2)}]  Conjecture \ref{conj Paszkiewicz} is true if the operators $T_1,T_2,\dots$ satisfy one of the following two conditions. 
        \begin{list}{}{}
            \item[{\rm (2-i)}] $1$ is not in the essential spectrum of $T_n$ for some $n\in \N$.
            \item[{\rm (2-ii)}] $\{T_1,T_2,\dots\}$ generates a finite von Neumann algebra. 
        \end{list} 
    \end{list} 
\end{theorem}
In Proposition \ref{prop: S_n*converges}~(1), which will be used to prove the equivalence of the above conjectures, we show that $\disp \lim_{n\to \infty}S_n^*=P$ (SOT) holds. Thus, (2-ii) is an immediate consequence of this convergence. Note also that because the $*$-operation is continuous on norm-bounded sets in the WOT (weak operator topology), the WOT analogue of Paszkiewicz's conjecture is true:
\begin{proposition}
Let $T_1\ge T_2\ge \dots$ be as in Conjecture \ref{conj Paszkiewicz}. 
Then $\disp \lim_{n\to \infty}S_n=P$ weakly. 
\end{proposition}
However, on an infinite-dimensional Hilbert space, the $*$-operation is highly discontinuous on norm-bounded sets in the SOT. In this context, we remark that by the classical Amemiya--Ando's theorem, random products of projections always converge in the WOT (see \cite{MR0187116AmemiyaAndo} for details and more general results), while the SOT-convergence of random products of projections was shown not to hold in general by Paszkiewicz \cite{Paszkiewiczpreprint2012} for random products of 5 projections, and later by Kopeck\'{a} and M\"{u}ller \cite{MR3268722KopeckaMuller14} for random products of 3 projections (see also \cite{MR3642022KopeckaPaszkiewicz2017}). Thus, the difference between the SOT and the WOT is significant. 
\section{Proof of Theorem~\ref{thm main}}
For general facts about self-adjoint operators and spectral theory, we refer the reader to \cite{MR2953553Schmudgenbook}. In the sequel, we fix a separable infinite-dimensional Hilbert space $H$. The set of all bounded linear operators on the Hilbert space $H$ is denoted by $\mathbb{B}(H)$. 
\subsection{Proof of Theorem \ref{thm main}~(1) and (2-ii)}
First, we show Theorem \ref{thm main} (1) and (2-ii). 
The following elementary lemma will be useful. 
\begin{lemma}\label{lem contraction eigenvalue}Let $T\in \mathbb{B}(H)$ be a positive contraction. 
    Then for $\xi\in H$, the following three conditions are equivalent. 
    \begin{list}{}{}
        \item[{\rm (1)}] $T\xi=\xi$.
        \item[{\rm (2)}] $\|T\xi\|=\|\xi\|$.
        \item[{\rm (3)}] $\nai{T\xi}{\xi}=\|\xi\|^2$.   
    \end{list}
    In particular, if $T,T'\in \mathbb{B}(H)_+$ satisfies $0\le T'\le T\le 1$, 
    then $1_{\{1\}}(T')\le 1_{\{1\}}(T)$ holds. 
\end{lemma}
\begin{proof}
    For the first part, (1)$\implies$(2) and (1)$\implies$(3) are clear. We show (2)$\implies$(1). Assume (2).
    For $0<\varepsilon<1$, let $p_{\varepsilon}=1_{[1-\varepsilon,1]}(T)$. Then 
    \eqa{
        \|\xi\|^2=\|T\xi\|^2&=\|Tp_{\varepsilon}\xi\|^2+\|Tp_{\varepsilon}^{\perp}\xi\|^2\\
        &\le \|p_{\varepsilon}\xi\|^2+(1-\varepsilon)^2\|p_{\varepsilon}^{\perp}\xi\|^2
    }
    This implies, because $\|\xi\|^2=\|p_{\varepsilon}\xi\|^2+\|p_{\varepsilon}^{\perp}\xi\|^2$, that $p_{\varepsilon}^{\perp}\xi=0$. Thus $\xi=p_{\varepsilon}\xi$. Since $\disp \lim_{\varepsilon\to +0}p_{\varepsilon}=1_{\{1\}}(T)$ strongly, we obtain $\xi=1_{\{1\}}(T)\xi$, whence $T\xi=\xi$ holds.\\
    Finally, we show (3)$\implies$(1). By (3), we have $\|T^{\frac{1}{2}}\xi\|=\|\xi\|$, which by (2)$\implies$(1) applied to $T^{\frac{1}{2}}$ implies that $T^{\frac{1}{2}}\xi=\xi$. Thus $T\xi=\xi$ holds.   

    For the last part, if $\xi\in H$ satisfies $T'\xi=\xi$, then by $T'\le T$, $\|\xi\|^2=\nai{T'\xi}{\xi}\le \nai{T\xi}{\xi}\le \|\xi\|^2$. Therefore, $\nai{T\xi}{\xi}=\|\xi\|^2$ holds. By (3)$\implies$(1), we obtain $T\xi=\xi$. This shows that $1_{\{1\}}(T')\le 1_{\{1\}}(T)$.  
\end{proof}
\begin{corollary}\label{cor Pn conv to P}
    Let $T_1\ge T_2\ge \dots$ be a decreasing sequence of positive contractions on a Hilbert space $H$. Then $P_n=1_{\{1\}}(T_n)$ converges to $P=1_{\{1\}}(T)$ in SOT, where $\disp T=\lim_{n\to \infty}T_n$ (SOT). 
\end{corollary}
\begin{proof}
    By Lemma~\ref{lem contraction eigenvalue}, we know that $P_1\ge P_2\ge \dots$ is a non-increasing sequence of projections. Therefore the SOT-limit $P'=\lim_{n\to \infty}P_n$ exists and $P'$ is also a projection. Since $T_n\ge T$ for every $n\in \N$, we have $P_n\ge P$ by Lemma~\ref{lem contraction eigenvalue}, whence $P'\ge P$ holds. If $\xi\in P'(H)$, then for every $n\in \N$, $P_n\ge P'$, whence $P_n\xi=\xi$. Therefore $T_n\xi=\xi$. Letting $n\to \infty$, we obtain $T\xi=\xi$, which shows that $\xi\in P(H)$. Therefore $P'\le P$, and $P'=P$ holds. 
\end{proof}
Proof of Theorem~\ref{thm main}~(1) follows from parts (1) and (2) of the next proposition (part (3) is not needed, but we include the observation which might be useful for further study).  
\begin{proposition}\label{prop: S_n*converges}Let $T_1\ge T_2\ge \cdots$ be a sequence of positive contractions on $H$ and let $S_n:=T_{n}\cdots T_1$. 
    The following statements hold (WOT stands for the weak operator topology): 
    \begin{list}{}{}
    \item[{\rm{(1)}}] $\disp \lim_{n\to \infty}S_n^*=P$ (SOT). In particular, $\disp \lim_{n\to \infty}S_n=P$  (WOT) holds.
    \item[{\rm{(2)}}] Let $\xi\in H$. If the set $\{S_n\xi\mid\,n\in \mathbb{N}\}$ is totally bounded, then $\disp \lim_{n\to \infty}\|S_n\xi-P\xi\|=0$ holds. 
    \item[{\rm{(3)}}] For every $\xi\in H$ and every $k\in \mathbb{N}$, $\displaystyle \lim_{n\to \infty}\|S_{n+k}\xi-S_n\xi\|=0$ holds. 
    \end{list}
    \end{proposition}
    \begin{proof}(1) Let $\xi\in H$ and $k\in \mathbb{N}$. Then for each $n\in \mathbb{N}$, 
    \[\|S_{n+k}^*P^{\perp}\xi\|=\|T_1\cdots T_n(T_{n+1}\cdots T_{n+k}P^{\perp}\xi)\|\le \|T_{n+1}\cdots T_{n+k}P^{\perp}\xi\|.\]
    Since $T_{n+j}\stackrel{n\to \infty}{\to}T\ (1\le j\le k)$ (SOT) and since the operator multiplication is jointly SOT-continuous on the unit ball of $\mathbb{B}(H)$, 
    $T_{n+1}\cdots T_{n+k}\stackrel{n\to \infty}{\to}T^k$ (SOT) holds. Thus 
    \[\limsup_{n\to \infty}\|S_n^*P^{\perp}\xi\|=\limsup_{n\to \infty}\|S_{n+k}^*P^{\perp}\xi\|\le \|T^kP^{\perp}\xi\|.\]
    By letting $k\to \infty$, we obtain 
    \[\|T^{k}P^{\perp}\xi\|^2=\int_{[0,1)}t^{2k}d\|e_T(t)\xi\|^2\stackrel{k\to \infty}{\to}0\]
    by the Lebesgue dominated convergence theorem. Here, $e_T(\cdot)$ is the spectral resolution of $T$. This shows that $\lim_{n\to \infty}\|S_n^*P^{\perp}\xi\|=0$. On the other hand, $T_nP\xi=P\xi$, for every $n\in \mathbb{N}$, whence $\lim_{n\to \infty}S_n^*P\xi=P\xi$. Hence $\lim_{n\to \infty}\|S_n^*\xi-P\xi\|=0$.\\
    (2) Let $\varepsilon>0$ and $k\in \mathbb{N}$. Since $\{S_n\xi\mid\,n\in \mathbb{N}\}$ is totally bounded and since $S_nP\xi=P\xi\ (n\in \mathbb{N})$, the set $\{S_nP^{\perp}\xi\mid n\in \mathbb{N}\}$ is totally bounded. Thus, there exists $n_0\in \mathbb{N}$ such that for every $n\in \mathbb{N}$, there exists $k_n\in \{1,\dots, n_0\}$ such that $\|S_nP^{\perp}\xi-S_{k_n}P^{\perp}\xi\|<\varepsilon$ holds. Then 
    \eqa{
    \|S_{n+k}P^{\perp}\xi\|&=\|T_{n+k}\cdots T_{n+1}S_nP^{\perp}\xi\|\\
    &\le \|T_{n+k}\cdots T_{n+1}(S_nP^{\perp}\xi-S_{k_n}P^{\perp}\xi)\|+\|T_{n+k}\cdots T_{n+1}S_{k_n}P^{\perp}\xi\|\\
    &\le \|S_nP^{\perp}\xi-S_{k_n}P^{\perp}\xi\|+\|T_{n+k}\cdots T_{n+1}S_{k_n}P^{\perp}\xi\|.
    }
    By a similar reasoning as in (i), we get 
    \[\limsup_{n\to \infty}\|S_nP^{\perp}\xi\|=\limsup_{n\to \infty}\|S_{n+k}P^{\perp}\xi\|\le \varepsilon+\max_{1\le j\le n_0}\|T^kS_jP^{\perp}\xi\|.\]
    Since $S_jP^{\perp}\xi\in P^{\perp}(H)\ (1\le j\le n_0)$, it follows that $\lim_{k\to \infty}\max_{1\le j\le n_0}\|T^kS_jP^{\perp}\xi\|=0$. Thus $\limsup_{n\to \infty}\|S_nP^{\perp}\|\le \varepsilon$. Since $\varepsilon>0$ is arbitrary, it follows that $\lim_{n\to \infty}\|S_nP^{\perp}\xi\|=0$. Therefore $\lim_{n\to \infty}\|S_n\xi-P\xi\|=0$ holds.\\
    (3) It suffices to show that $\lim_{n\to \infty}\|S_{n+1}\xi-S_n\xi\|=0$. Define $a_n=\nai{S_{n+1}\xi}{S_{n}\xi}$ and $b_n=\nai{S_n\xi}{S_n\xi}\ (n\in \mathbb{N})$. Since $T_{n+1}$ is a positive contraction, we have $0\le a_n\le b_n$. 
    Since $T_n$ is a contraction, $(b_n)_{n=1}^{\infty}$ is positive and non-increasing. Therefore, the limit $\beta=\lim_{n\to \infty}b_n$ exists. On the other hand, for every $n\in \mathbb{N}$, 
    \eqa{
    b_{n+1}&=\nai{S_{n+1}\xi}{S_{n+1}\xi}=\nai{T_{n+1}^2S_n\xi}{S_n\xi}\\
    &\le \nai{T_{n+1}S_n\xi}{S_n\xi}=a_n\le \nai{S_n\xi}{S_n\xi}=b_n.
    }
    This implies that $\lim_{n\to \infty}a_n=\beta$ holds. Then 
    \eqa{
    \|S_{n+1}\xi-S_n\xi\|^2&=\|S_{n+1}\xi\|^2+\|S_n\xi\|^2-2{\rm{Re}}\nai{S_{n+1}\xi}{S_n\xi}\\
    &\stackrel{n\to \infty}{\to}\beta+\beta-2\beta=0.
    }
    This shows the claim. 
    \end{proof}

\begin{remark}
Let $(\xi_n)_{n=1}^{\infty}$ be a sequence in $H$ with the following properties. 
\begin{list}{}{}
\item[{\rm{(i)}}] $\lim_{n\to \infty}\xi_n=0$ weakly in $H$. 
\item[{\rm{(ii)}}] $(\|\xi_n\|)_{n=1}^{\infty}$ is non-increasing (hence it is convergent).
\item[{\rm{(iii)}}] For every $k\in \mathbb{N}$, $\lim_{n\to \infty}\|\xi_{n+k}-\xi_n\|=0$ holds.  
\end{list}
If it follows that $\{\xi_n\mid n\in \mathbb{N}\}$ is totally bounded,  then $\disp \lim_{n\to \infty}S_n=P$ (SOT) by Proposition \ref{prop: S_n*converges} (ii) (put $\xi_n=S_nP^{\perp}\xi$). We remark, however, that the set $\{\xi_n\mid n\in \mathbb{N}\}$ satisfying (i), (ii) and (iii) need not be totally bounded in general. 
\end{remark}

\begin{nonexample}
Let $\theta_n=\tfrac{\pi}{2n}\ (n\in \mathbb{N})$. 
Fix a CONS $(e_n)_{n=1}^{\infty}$ for $H$. We will construct a sequence $\mathcal{S}=\{\eta_{n,j}\mid n\in \mathbb{N},\,1\le j\le n\}$ of unit vectors in $H$ with the following properties: 
\begin{list}{}{}
\item[{\rm{(a)}}] $\|\eta_{n,j+1}-\eta_{n,j}\|=\|\eta_{n+1,1}-\eta_{n,n}\|=2\sin \frac{\theta_n}{2} (n\ge 2,\,1\le j\le n-1)$. 
\item[{\rm{(b)}}] $\eta_{n,1}=e_n\ (n\in \mathbb{N})$. In particular, $\{e_n\mid n\in \mathbb{N}\}\subset \mathcal{S}$ holds. 
\item[{\rm{(c)}}] $\eta_{n,j}\in {\rm{span}}\{e_n,e_{n+1}\}\ (n\in \mathbb{N},\,1\le j\le s_n)$. 
\end{list}
Define a linear ordering $<$ on $I=\{(n,j)\mid n\in \mathbb{N},\,1\le j\le n\}$ by $(n,j)<(n',j')$ if $n<n'$ or $n=n'$ and $j<j'$. Fix the order-preserving bijection $\nai{\,\cdot\,}{\,\cdot\,}\colon I\to \mathbb{N}$ given by
\[\nai{n}{j}=j+\frac{(n-1)n}{2},\ \ 1\le j\le n.\]
We set
\[\eta_{n,j}:=\cos ((j-1)\theta_n)e_n+\sin ((j-1)\theta_n)e_n,\ \ n\in \mathbb{N},\,1\le j\le n\]
By construction, (b) and (c) hold. 
Both the angle between $\eta_{n,j+1}$ and $\eta_{n,j}$ and the angle between $\eta_{n+1,1}$ and $\eta_{n,n}$ are $\theta_n$ if $n\ge 2$ and $1\le j\le n$. Therefore (a) holds.

Then the sequence $(\xi_n)_{n=1}^{\infty}$ given by $\xi_{\nai{n}{j}}=\eta_{n,j}\ ((n,j)\in I)$ does the job. Indeed, by (a), (iii) holds. It is obvious that (ii) holds. (i) holds because of (c). However, by (b), $\{\xi_n\}_{n=1}^{\infty}$ is not relatively compact, hence it is not totally bounded.  Note that the sequence $(\xi_n)_{n=1}^{\infty}$ we constructed is of the form 
\[\xi_n=U_nU_{n-1}\cdots U_1e_1,\]
where each $U_n$ is a unitary such that ${\rm{rank}}(U_n-1)=2$ for every $n\in \mathbb{N}$.
\end{nonexample} 
    \begin{proof}[Proof of Theorem~\ref{thm main}~(1) and (2-ii)]
        It is clear that Conjecture~\ref{conj strong Paszkiewicz}$\implies$ Conjecture~\ref{conj Paszkiewicz} holds. Conversely, assume that Conjecture~\ref{conj Paszkiewicz} holds and $\disp S=\lim_{n\to \infty}S_n$ (SOT). 
        Then for each $\xi\in H$, $(S_n\xi)_{n=1}^{\infty}$ converges, whence it is totally bounded. Therefore by Proposition~\ref{prop: S_n*converges} (i) and (ii), $\disp \lim_{n\to \infty}S_n=P$ (S$^*$OT). Therefore Conjecture \ref{conj strong Paszkiewicz} holds. This proves (1).\\ 
        (2-ii) follows from Proposition \ref{prop: S_n*converges}\,(i) and the fact that the map $V\mapsto V^*$ is strongly continuous on the unit ball of a finite von Neumann algebra. 
     \end{proof}
\subsection{Proof of Theorem \ref{thm main}~(2)}
Here we prove Theorem \ref{thm main}~(2). 
\begin{definition}
    Let $T_1\ge T_2\ge \dots$ be a decreasing sequence of positive contractions on a Hilbert space $H$. We say that it has uniform spectral gap at 1, if there exist $\delta\in (0,1)$ and $N\in \mathbb{N}$ such that $\sigma(T_n)\cap (1-\delta,1)=\emptyset$ for all $n\ge N$.
\end{definition}
Theorem \ref{thm main}~(2) follows from the next Proposition. 
\begin{proposition}\label{prop: spectral gap}
Let $T_1\ge T_2\ge \dots$ be a decreasing sequence of positive contractions on a Hilbert space $H$. 
\begin{list}{}{} 
    \item[{\rm (1)}] If $T_1\ge T_2\ge \cdots$ has uniform spectral gap at 1.  Then the Conjecture \ref{conj Paszkiewicz} holds for $T_1\ge T_2\ge \cdots$. 
    \item[{\rm (2)}] If 1 is not in the essential spectrum of $T_n$ for some $n\in \N$, then $T_1\ge T_2\ge \cdots$ has uniform spectral gap at 1.
\end{list}
\end{proposition}

\begin{proof}[Proof of Proposition \ref{prop: spectral gap}~(1)]
Fix $\delta$ and $N$ witnessing the uniform spectral gap at 1 of $T_1\ge T_2\ge \dots$. 
Let $\xi\in H$. Assume first that $\xi \in P(H)$. Then for all $n\in \mathbb{N}$, $T_n\xi=\xi$, so that $S_n\xi=\xi$. Thus $\lim_{n\to \infty}S_n\xi=\xi=P\xi$. Next, assume that $\xi\in P^{\perp}(H)$ and $\varepsilon>0$. Choose $n_0\in \mathbb{N}$ such that $n_0\ge N$ and $\|\xi-P_{n_0}^{\perp}\xi\|<\varepsilon$.   Let $\eta:=T_{n_0}\cdots T_1\xi$ and $\eta':=T_{n_0}\cdots T_1P_{n_0}^{\perp}\xi$. Then because all $T_i'$s are contractions, we have $\|\eta-\eta'\|\le \|\xi-P_{n_0}^{\perp}\xi\|<\varepsilon$.\\
Note also that $\eta'\in P_{n_0}^{\perp}(H)$, because all $T_1,\dots, T_{n_0}$ leave the range of $P_{n_0}$ invariant, hence the range of $P_{n_0}^{\perp}$ invariant.
Since $P_{n_0}^{\perp}\le P_{n_0+1}^{\perp}$, we have $\eta'\in P_{n_0+1}^{\perp}(H)$. 
Thus $T_{n_0+1}\eta'\in P_{n_0+1}^{\perp}(H)\subset P_{n_0+2}^{\perp}(H)$, so that $T_{n_0+2}T_{n_0+1}\eta'\in P_{n_0+3}^{\perp}(H)$. By induction, we obtain 
\[\eta'_j:=T_{n_0+j}T_{n_0+j-1}\cdots T_{n_0+1}\eta'\in P_{n_0+j+1}^{\perp}(H)=1_{[0,1-\delta]}(T_{n_0+j+1})(H),\ \ \ j\in \mathbb{N}.\]
Therefore 
\eqa{
\|T_{n_0+j}(\underbrace{T_{n_0+j-1}\cdots T_{n_0+1}\eta'}_{=\eta'_{j-1}\in P_{n_0+j}^{\perp}(H)})\|&\le \|T_{n_0+j}|_{P_{n_0+j}^{\perp}(H)}\|\,\|T_{n_0+j-1}\eta_{j-2}'\|\\
&\le \cdots \le \prod_{k=1}^j\|T_{n_0+k}|_{P_{n_0+k}^{\perp}(H)}\|\,\|\eta'\|\\
&\le (1-\delta)^j\|\eta'\|.
}
This shows that 
\eqa{
\|S_{n_0+j}\xi\|&=\|T_{n_0+j}\cdots T_{n_0+1}\eta\|\\
&\le \|T_{n_0+j}\cdots T_{n_0+1}(\eta-\eta')\|+(1-\delta)^j\|\eta'\|\\
&<\varepsilon+(1-\delta)^j\|\eta'\|.
}
Thus $\limsup_{j\to \infty} \|S_{n_0+j}\xi\|\le \varepsilon$. Since $\varepsilon>0$ is arbitrary, we get $\displaystyle \lim_{n\to \infty}\|S_n\xi\|=0$. 
Therefore for general $\xi\in H$, we have 
\[\|S_n\xi-P\xi\|\le \|S_n(P\xi)-P\xi\|+\|S_nP^{\perp}\xi\|\xrightarrow{n\to \infty}0.\]
This shows that $\disp \lim_{n\to \infty}S_n=P$ (SOT). 
\end{proof}
For the proof of Proposition \ref{prop: spectral gap}~(2), we need the following lemma. We denote by $\sigma_{\rm e}(T)$ the essential spectrum of an operator $T$. 
\begin{lemma}\label{lem: gap and rank}
    Let $T,T'$ be positive contractions on a Hilbert space $H$ such that $T\ge T'$. 
    \begin{list}{}{}
        \item[{\rm (1)}] If $\delta\in (0,1)$ satisfies $\sigma(T)\cap (1-\delta,1)=\emptyset$ and $\sigma(T')\cap (1-\delta,1)\neq \emptyset$, then $P'\lneq P$ holds, where $P:=1_{\{1\}}(T)$ and $P':=1_{\{1\}}(T')$.
        \item[{\rm (2)}] If $1\in \sigma_{\rm e}(T')$, then $1\in \sigma_{\rm e}(T)$ holds.   
    \end{list}
\end{lemma}
\begin{proof}
    (1) By $0\le T'\le T\le 1$ and Lemma \ref{lem contraction eigenvalue}, we know that $P'\le P$ holds. Assume by contradiction that $P'=P$ holds. Let $t\in \sigma(T')\cap (1-\delta,1)$ and $\varepsilon>0$ be such that $1-\delta<t-\varepsilon$ and $t+\varepsilon<1$. Then there exists a nonzero vector $\xi\in 1_{(t-\varepsilon,t+\varepsilon)}(T')(H)$.  Since $P'=1_{\{1\}}(T')$ and $1_{(t-\varepsilon,t+\varepsilon)}(T')$ are orthogonal, we have $P'\xi=P\xi=0$.  This implies that $\xi=P^{\perp}\xi=1_{[0,1-\delta]}(T)\xi$. Then by $T'\le T$, we obtain 
    \begin{equation}
    (t-\varepsilon)\|\xi\|^2\le \nai{T'\xi}{\xi}\le \nai{T\xi}{\xi}=\nai{T1_{[0,1-\delta]}(T)\xi}{\xi}\le (1-\delta)\|\xi\|^2,
    \end{equation}
    which contradicts the condition $t-\varepsilon>1-\delta$. Therefore $P'\lneq P$ holds.\\
    (2) By $1\in \sigma_{\rm e}(T')$, there exists an orthonormal sequence $(\xi_n)_{n=1}^{\infty}$ in $H$ such that $\disp \lim_{n\to \infty}\|T'\xi_n-\xi_n\|=0$. Thus $\disp \lim_{n\to \infty}\|T'\xi_n\|=1=\lim_{n\to \infty}\nai{T'\xi_n}{\xi_n}$ holds.  
    Then by $0\le T'\le T\le 1$, we have 
    \eqa{
        \|T\xi_n-\xi_n\|^2&=\|T\xi_n\|^2-2\nai{T\xi_n}{\xi_n}+\|\xi_n\|^2\\
        &\le 2-2\nai{T\xi_n}{\xi_n}\\
        &\le 2-2\nai{T'\xi_n}{\xi_n}\xrightarrow{n\to \infty}0.
    }
    Thus, by Weyl's criterion for the essential spectrum (see e.g., \cite[Proposition 8.11]{MR2953553Schmudgenbook}), $1\in \sigma_{\rm e}(T)$ holds. 
\end{proof}

\begin{proof}[Proof of Proposition \ref{prop: spectral gap}~(2)] 
Assume that $1\notin \sigma_{\rm e}(T_{n_0})$. 
Let $P_{n_0}=1_{\{1\}}(T_{n_0})$. By $1\notin \sigma_{\rm e}(T_{n_0})$, $1$ is not an accumulation point of the spectrum $\sigma(T_{n_0})$ of $T_{n_0}$, and it is not an eigenvalue of $T_{n_0}$ of infinite multiplicity either.  
Thus there exists $\delta_0\in (0,1)$ such that 
$\sigma(T_{n_0})\cap (1-\delta_0,1)=\emptyset$, and $d={\rm rank}(P_{n_0})$ is finite (possibly $d=0$). If there is no $n>n_0$ such that $\sigma(T_n)\cap (1-\delta_0,1)\neq \emptyset$, then $\delta=\delta_0$ and $N=n_0$ work. If there is such an $n>n_0$, let $n_1$ be the smallest such  number. Then by $\sigma(T_{n_1})\cap (1-\delta_0,1)\neq \emptyset, T_{n_1}\le T_1$ and Lemma~\ref{lem: gap and rank}~(1), we have ${\rm{rank}}(P_{n_1})<{\rm{rank}}(P_{n_0})=d<\infty$ (thus $d\neq 0$ if such $n_1$ exists). By Lemma \ref{lem: gap and rank}~(2), $1\notin \sigma_{\rm e}(T_n)$ for every $n\ge n_0$. Thus, by the above argument, we may find $0<\delta_1<\delta$ such that $\sigma(T_{n_1})\cap (1-\delta_1,1)=\emptyset$. If there is no $n>n_1$ such that $\sigma(T_n)\cap (1-\delta_1,1)\neq \emptyset$, we set $N=n_1$ and $\delta=\delta_1$. If there is such an $n>n_1$, let $n_2$ be the smallest such number, and find $0<\delta_2<\delta_1$ such that $\sigma(T_{n_2})\cap (1-\delta_2,1)=\emptyset$. Then ${\rm{rank}}(P_{n_2})<{\rm{rank}}(P_{n_1})$. Inductively,
 we find a sequence $n_1<n_2<\cdots$ and $\delta_1>\delta_2>\cdots$. These sequences must have the same length at most $d$. Let $k$ be the length of these sequences. Then $N=n_k$ and $\delta=\delta_k$ work.   
\end{proof}

As a corollary, we obtain the following result, which was in fact the earliest and motivational result in this project, shown to us by Yasumichi Matsuzawa. The author would like to thank him for sharing his proof.  

\begin{corollary}[Matsuzawa]\label{cor compact case}
Conjecture \ref{conj Paszkiewicz} holds if $T_n$ is compact for some $n\in N$. 
\end{corollary}

\section*{Acknowledgments}
The current work started as a joint work with Yasumichi Matsuzawa (Shinshu University) in 2018. However, he declined to be a coauthor and suggested that the current author write the paper alone.  

Nevertheless, we would like to emphasize that it was exactly his proof of Corollary \ref{cor compact case} which motivated us to generalize his proof, and the outcome is the main result of this paper. Therefore, his contribution to this work is essential. 
The author is supported by Japan Society for the Promotion of Sciences KAKENHI 20K03647. 
\bibliographystyle{siam}
\bibliography{references} 
\end{document}